\documentclass[11pt]{article}

\usepackage{amsmath}
\usepackage{amsfonts}
\usepackage{amssymb}
\usepackage{amsthm}
\usepackage{graphicx}
\usepackage{epstopdf}
\usepackage{mathrsfs}
\usepackage{amsbsy}
\usepackage{amsopn}
\usepackage{float}
\usepackage[all]{xy}
\usepackage{enumerate}
\usepackage{color}
\usepackage[T1]{fontenc}
\usepackage{ae,aecompl}
\usepackage{float}
\usepackage[utf8]{inputenc}

\theoremstyle{plain}
\newtheorem{teo}{Theorem}[section]
\newtheorem{lema}[teo]{Lemma}
\newtheorem{prop}[teo]{Proposition}
\newtheorem{crl}[teo]{Corolary}
\theoremstyle{definition}
\newtheorem{defn}[teo]{Definition}
\newtheorem{eje}{Example}
\newtheorem{prob}{Problem}

\begin{document}
\title{Gammoids, Pseudomodularity and Flatness Degree}
\author{Jorge Alberto Olarte\thanks{Universidad de los Andes, Bogot\'a, Colombia. ja.olarte1299@uniandes.edu.co \newline The author would like to thank Alf Onshuus, without his valuable comments and guidance this article would not have been possible.}}

\maketitle

\begin{abstract}
We study weaker variations of the property of flatness in matroid theory. We show that these variations form a chain of increasingly stronger properties all implying pseudomodularity on its lattice of flats. We show examples in the gammoid class that show that these properties are in fact different.
\end{abstract}

\section{Introduction}

Hrushovski introduces the concept of a matroid being \emph{flat} in \cite{hrushovski1993new} in order to prove the existence of a non trivial strongly minimal set that does not interpret an infinite field. In \cite{evans2011matroid}, Evans shows that for finite matroids the notion of flatness characterizes the matroids known as \emph{strict gammoids} that were first studied by Mason \cite{mason1972class}. Strict gammoids are also known as \emph{cotransversals}, as Ingleton and Piff showed that they are precisely the duals of transversal matroids \cite{ingleton1973gammoids}. Actually, in \cite{mason1971characterization} Mason gives the exact dual analog of flatness as a characterization of transversal matroids. The restrictions of a strict gammoid, are known as simply \emph{gammoids} and they form a complete class of matroids that has been widely studied \cite{brualdi1987transversal} \cite{ingleton1977transversal} \cite{oxley1992matroid}.

In \cite{evans2011matroid}, Evans shows that strict gammoids have a \emph{pseudomodular} lattice of flats. The notion of pseudomodularity was first studied by Dress and Lov\'asz in \cite{dress1987some} as a necessary condition for full algebraic matroids and was formally defined in \cite{bjorner1987pseudomodular} by Björner and Lov\'asz. Evans asks whether a gammoid which is pseudomodular is necessarily a strict gammoid. We answer negatively by showing that a strictly weaker condition than flatness, having \emph{flatness degree} at least 3 is enough for pseudomodularity. We then construct a gammoid showing that the converse is not true, that is being pseudomodular and having flatness degree 2. For each possible flatness degree we construct a gammoid having such flatness degree. All of these are counter-examples to Evans' question. We get an infinite chain of increasingly stronger properties, from pseudomodularity to flatness, all of them speaking in terms of the lattice of flats.

We provide basic knowledge about matroids and the gammoid class in section 2. We define and discuss pseudomodularity in section 3 and we connect it with the notion of flatness. In section 4 we define the flatness degree and construct examples of gammoids that attain each of the possible values for flatness degree. Finally we propose some problems in section 5.

\section{Matroid theory background}

We do not assume any previous knowledge of matroid theory, so we review the basic concepts in this section. Proofs and further insight can be found in \cite{oxley1992matroid} and, for the specific gammoid class, in  \cite{brualdi1987transversal} and \cite{brualdi1987introduction}. Throughout the paper we consider only finite matroids although the results can be extended to any matroid with finite rank unless it is specifically stated that the matroid must be finite. A \emph{matroid} $M$ consists of a (finite) set $N$ and a function $r : \mathcal{P}(N) \rightarrow \mathbb{Z}$ called rank that satisfy the following conditions

\begin{enumerate}[(R1)]
	\item If $A \subseteq N$ then $0 \leq r(A) \leq |A|$
	\item If $A \subseteq B \subseteq N$ then $r(A) \leq r(B)$
	\item If $A \subseteq N$ and $B \subseteq N$ then $r(A) + r(B) \geq r(A\cup B) + r(A\cap B)$
\end{enumerate}

The last condition is called \emph{submodularity}.  The rank of a matroid is $r(N)$. A subset $I$ of $N$ is called \emph{independent} if $r(I) = |I|$. A subset $B$ of $N$ is called a \emph{basis} if it is independent and $r(B) = r(N)$. A subset $C$ of $N$ is a \emph{circuit} if it is a minimal dependent set. That is, for every $n \in C$, $r(C) = r(C\setminus\{e\}) = |C|-1$. A subset $F$ of $N$ is called flat if for every $n \in N \setminus F$ we have $r(F\cup \{n\}) > r(F)$. We will write $\mathcal{F} = \mathcal{F}(M)$ as the set of flats of the matroid $M$. When the equality in (R3) is met whenever $A$ and $B$ are flat, we say that $M$ is \emph{modular}. We can define the closure operator $cl : \mathcal{P}(N) \rightarrow \mathcal{P}(N)$ as $cl(A) := \min (\{F \in \mathcal{F} \enspace | \enspace A \subseteq F \})$ or, equivalently $cl(A) = max(\{F \subseteq N \enspace | \enspace r(A) = r(F) \})$. 

\let\thefootnote\relax\footnote{We chose to use the classical matroid notation. Sometimes, the rank function is called \emph{dimension} and noted by $d$ and flats are called \emph{closed} sets. The use of term flat for closed sets may cause confusion with the property for matroids also called flat. That is why we choose to refer to this property as a matroid being \emph{totally flat}. We hope that the difference between the object flat and any property with flat in its name is expressed clearly enough.}

Note that the rank of $A$ tells the maximum cardinality of an independent subset of $A$. Then the matroid is determined by the set of independent subsets of $N$. Also, the rank of a flat $F$ is the length of a maximal chain of flats such that $F_0 \subset  F_1 \subset \dots \subset F_{r(F)} = F$. As the rank of any set is the rank of its closure, which is a flat, the matroid is also determined by the set of flats. Actually matroids can be criptomorphically defined by all of the objects defined above. We choose to use the rank function for the definition, because of the importance of submodularity in this paper. So sometimes we will refer to a matroid by giving its set of independent sets or the set of flats rather than giving the duple $(N, r)$.

Given a matroid $M$ on the set $N$, we can define naturally matroids on subsets of $N$. More precisely, for $A \subseteq N$, we define the \emph{restriction} of $A$ as the matroid $M\setminus A$ on the set set $N\setminus A$ with rank function $r_{M\setminus A}$ as the rank function restricted to $\mathcal{P}(N\setminus A)$. We can also define the \emph{contraction} of $A$ as the matroid $M/A$ which also has set $N\setminus A$ but with rank function $r_{M\setminus A}(B) := r(B\cup A) - r(A)$ for any $X \subseteq N \setminus A$. The operations of restricting and contracting a matroid commute, this is $(M \setminus A) /B = (M/B)\setminus A$. Any matroid constructed this way is called a \emph{minor} of $M$.

Finally we can define the \emph{dual} $M^*$ of $M$, acting on the same set $N$ but with rank function $r^*(A) := |A|+r(N\setminus A) -r(N)$. It can be easily verified that the bases of $M^*$ are precisely the complement of the bases of $M$. Thus, $(M^*)^* = M$ and hence the term dual. It can be shown that if $A \subseteq N$ then $(M\setminus A)^* = M^*/A$. Note that this implies $M^*\setminus A = (M/A)^*$. So in this sense we can say that the restriction and the contraction are dual. There are many more objects that can be said to be dual in some sense. An element $n \in N$ is said to be a \emph{loop} if $r(n) = 0$. An element $n \in N$ is a \emph{coloop} if it is a loop in $M^*$ or, equivalently, if for all $A \subseteq N$ $A$ is independent if and only if $A\cup\{n\}$ is independent. A subset $S$ of $N$ is called \emph{cyclic} if it does not for every $n \in S$, $r(S\setminus \{n\}) = r(S)$. In other words, $S$ is cyclic if $M \setminus (N \setminus S)$ does not have coloops. Note that all circuits are cyclic. It is easy to verify that $S$ is cyclic if and only if $N\setminus S$ is a flat.

We now turn our attention to the class of gammoids, starting with transversal matroids. Given $\mathcal{A} = (A_1, \dots ,A_k)$ subsets of $N$, and a subset $I$ of $N$, a \textit{matching} of $I$ is an injective function from $I$ to $\mathcal{A}$ such that for every $n \in I$, we have $n \in f(n)$. The set of subsets of $N$ that have a matching form the independent sets of a matroid $M$. The set $\mathcal{A}$ is called a \textit{presentation} of $M$. A matroid that can be constructed this way is called \emph{transversal}. Different presentations may produce the same matroid. It is obvious from the definition that the rank of $M$ is at most $k$, however, one can always give presentations such that $k$ is exactly the rank of $M$.

Now let $\Gamma = (N, E)$ be a directed graph where $N$ is the set of vertices and $E$ is the set of edges. Given $u, v \in N$, A \emph{path} $P$ from $u$ to $v$ is a sequence of vertices $(n_1, \dots, n_t)$ such that $n_1 = u$, $n_t = v$ and for every $i \leq t-1$, $(n_i, n_{i+1}) \in E$. Given two sets $A, B \subseteq E$, a \emph{linking} $\Theta$ from $A$ to $B$ is a collection of paths such that

\begin{itemize}
	\item $|\Theta| = |A|$
	\item The paths are pairwise disjoint
	\item Each path starts in a vertex belonging to $A$ and ends in a vertex of $B$
\end{itemize}

Fixing a sets $N_1, N_2 \subseteq N$, the sets $I \subseteq N_1$ such that there is a linking from $I$ to $N_2$ are the independent sets of a matroid $M$ on the set $N_1$. A matroid that can be constructed this way is called a \emph{gammoid}. If $N_1 = N$ we say it is a \textit{strict gammoid}. A gammoid can have different directed graphs representing it.

The gammoid class closed under minors and duals. Strict gammoids and transversal matroids are known to be dual \cite{ingleton1973gammoids}. It is clear from the definition that transversal matroids are closed under restriction, however they are not closed under contraction. Dually, strict gammoids are closed under contraction but not closed under restriction. Every gammoid is the restriction of a strict gammoid. Then every gammoid is the contraction of a transversal matroid. 

\begin{eje}
\label{matroidY}

\begin{figure}[h]
\centering	
	\caption{}
	\includegraphics{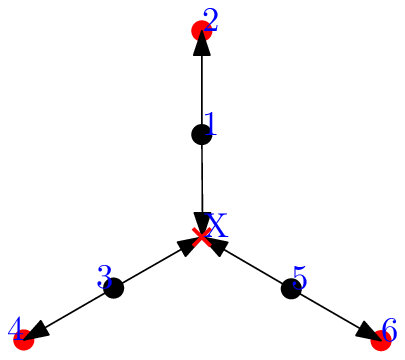}
	\label{fig:ej1}
	
\end{figure}

Let $\Gamma$ be the graph shown in Figure \ref{fig:ej1}. By letting $N_1 = \{1,2,3,4,5,6\}$ and $N_2 = \{2, 4, 6, x\}$  we get a gammoid $M$ of rank $4$. This gammoid is also a transversal matroid, as $\{\{1,2\}, \{3,4\},\{5,6\}, \{1,3,5\}\}$. is a presentation of $M$. However it will be shown in section $3$ that $M$ is not a strict gammoid (this does not follow from the fact that $x \notin N_1$, as there may be a different directed graph representing $M$ with all of its vertices).
\end{eje}

\section{3-Flatness and Pseudomodularity}

There have been many characterizations for transversal matroids and strict gammoids, most of them speaking in terms of flats and specially cyclic flats. However, the one we like the most for strict gammoids is Theorem \ref{charac}. Mason first showed the analogue for transversal matroids in \cite{mason1971characterization}. However, it is specifically stated and shown by Evans in \cite{evans2011matroid}. 

Let $M$ be a matroid. Recall that we refer as $\mathcal{F}$ as the set of flats of $M$. Let $\mathcal{C} = \{F_i \in \mathcal{F} \enspace | \enspace i \in I\}$ a collection of flats. If $\emptyset \neq S\subseteq I$ let $F_S= \bigcap_{i\in S}F_i$ and $F_\emptyset = \bigcup\limits_{i\in I} F_i$. We define the $\Delta$ function as
\begin{equation*}
\Delta (\mathcal{C}) = \sum_{S\subseteq I} (-1)^{|S|}r(F_s) 
\end{equation*}

For such $\mathcal{C}$ we want to admit repetitions even though they are trivial for calculating $\Delta(\mathcal{C})$, thus letting $\mathcal{C}$ be a multiset. We want to do this, because we are going to manipulate the elements in $\mathcal{C}$ and want to maintain the cardinality of $\mathcal{C}$ cardinality constant. The $\Delta$ function can help us define the concept of flatness given by Hrushovski in \cite{hrushovski1993new}.

\begin{defn}
We say $M$ is \emph{totally flat} if $\forall \mathcal{C} \subseteq \mathcal{F}$ finite, $\Delta (\mathcal{C}) \leq 0$
\end{defn}

\begin{teo}
\label{charac}
A finite matroid $M$ is a strict gammoid if and only if it is totally flat. 
\end{teo}

Consider Example \ref{matroidY}. Let $\mathcal{C} = \{\{1,2,3,4\},\{1,2,5,6\},\{3,4,5,6\}\}$ we have $\Delta(\mathcal{C}) = 4-3-3-3+2+2+2=1$ so $M$ is not totally flat. Hence, it is not a strict gammoid. It is actually the minimal gammoid that is not strict gammoid, in terms of rank and size. As stated before, we want to keep track of the cardinality of $\mathcal{C}$ so we can study the following properties

\begin{defn}
A matroid $M$ is \emph{$n$-flat}, if and only if $\forall \mathcal{C} \subseteq \mathcal{F}$ such that $|\mathcal{C}| \leq n$, then $ \Delta(\mathcal{C}) \leq 0$.
\end{defn}

We used the term $n$-flat for $|\mathcal{C}| \leq n$ instead of $|\mathcal{C}| = n$ because once $\Delta(C) > 0$ it is easy to generate collections of sets of bigger cardinality than $\mathcal{C}$ by adding any subflat of elements of $\mathcal{C}$ (see Proposition \ref{repet}). Note that 2-flat is equivalent to submodularity, so all matroids are trivially 2-flat. The first non trivial property would be 3-flat, and it happens to imply several important properties.

Now let $A, B \in \mathcal{F}$. We note as $r(A/B) := r(A\cup B) - r(B)$, which is the rank of $A$ when contracting $B$. Suppose there exists $B_0 \in \mathcal{F}$ such that for every flat $B_1 \subseteq B$, $r(A/B_1) = r(A/B)$ if and only if $B_0 \subseteq B_1$. $B_0$ is called the \emph{pseudointersection} of $A$ and $B$. Note that the pseudointersection is not symmetric, that is, the pseudointersection of $A$ and $B$ can be different from the pseudointersection of $B$ and $A$. Further, not even the existence of pseudointersection is symmetric. Example of these may be seen in \cite{bjorner1987pseudomodular}. If the pseudointersection of $A$ and $B$ exists, we write $A \blacktriangleleft B$. If it does not exist, we write $A \not \blacktriangleleft B$.

\begin{defn}
A matroid $M$ is \emph{pseudomodular} if for every $A, B \in \mathcal{F}$ we have $A \blacktriangleleft B$ 
\end{defn}

The concepts of pseudointersection and pseudomodularity speak in terms of the lattice of flats of a matroid and can be extended for any lattice. It is not difficult to see that pseudomodularity may be restated as follows: for every $A, B, C \in \mathcal{F}$ if $r(A/B) = r(A/C) = r(A/ B\cup C)$ then $r(A/ B\cap C= r(A/B)$. Consider again Example \ref{matroidY}. If $A = \{1,2\}$, $B = \{3,4,5,6\}$, $B_1 = \{3,4\}$ and $B_2 = \{5,6\}$ we have $r(A/B) = r(A/B_1) = r(A/B_2) = 1$. But $B_1 \cap B_2 = \emptyset$, so $r(A/ \emptyset) = r(A) = 2$. Then $A \not \blacktriangleleft B$ and so $M$ is not pseudomodular. If we consider the strict gammoid generated by $\Gamma$, say $M_0$, $X$ would be in the closure of $B$, $B_1$ and $B_2$, so $B_1 \cap B_2 = \{X\}$ and we would have $r(A/ \{X\}) = 1$. $M_0$ is actually modular, as Evans shows in \cite{evans2011matroid} that strict gammoids are pseudomodular. We show Theorem \ref{3fwcb} as a generalization of that result. Before proving Theorem \ref{3fwcb} we need the following lemma

\begin{lema}
\label{wcb1}
Let $M$ be a matroid which is not pseudomodular. Then there exists $B \in \mathcal{F}$ and $A \subseteq E$ such that $r(A/B) = 1$ and $A \not \blacktriangleleft B$.
\end{lema} 
\begin{proof}
If $M$ is not pseudomodular, then there exists $A, B\in \mathcal{F}$ such that $A \not \blacktriangleleft B$. This means there are flats $B_1, B_2 \subseteq B$ such that $r(A/B) = r(A/B_1) = r(A/B_2) < r(A/B_1\cap B_2)$. We now proof the lemma by induction on $r(A/B)$. If $r(A/B) = 1$ then we have the desired result.
Now let $r(A/B) = n$ and suppose the lemma is true whenever there are $B'  \in \mathcal{F}$, $A' \subseteq E$ such that $A' \not \blacktriangleleft B'$ and $r(A'/B') = n-1$. Let $x \in A\backslash B$ and consider the set
\begin{equation*}
C = (cl(B_1\cup\{x\})\cap cl(B_1\cup\{x\}))\backslash (B_1\cap B_2) 
\end{equation*}

As $C \subseteq cl(B_1\cup\{x\})$, $r(C\cup B_1) = r(B_1)+1$ and so $r(C/B_1) =1$. In the same way $r(C/B_2) =1$. As $x \in C\backslash B$ then $r(C/B) \geq 1$ and as $B_1 \subseteq B$ then $r(C/B) \leq r(C/B_1) = 1$, so $r(C/B) = 1$. If $r(C/B_1\cap B_2) > 1$ then $C \not \blacktriangleleft B$ and we have what we want.
Suppose now $r(C/B_1\cap B_2) = 1$. Note that $C\cup (B_1\cap B_2) = cl(B_1\cup\{x\})\cap cl(B_2\cup\{x\})$. Let $B' = cl(B\cup \{x\})$, $B_i' = cl(B_i \cup \{x\})$ for $i \in \{1,2\}$. As $x \in A$, $cl(A\cup B') = cl(A\cup B)$. Then 
\begin{eqnarray*}
r(A/B') &=& r(A\cup B') -r(B')\\ 
&=& r(A\cup B)-r(B) -(r(B') - r(B))\\ 
&=& r(A/B) -r(B'/B) = n-1
\end{eqnarray*}

In the same way $r(A/B_1') = r(A/B_2') = n-1$. As $B_1' \cap B_2' = C\cup (B_1 \cap B_2)$, we have:
\begin{eqnarray*}
r(A/B_1'\cap B_2') &=& r(A\cup(B_1'\cap B_2')) - r(B_1'\cap B_2') \\
&=& r(A\cup(B_1'\cap B_2')) -r(B_1\cap B_2)\\
&-&(r(B_1'\cap B_2')- r(B_1\cap B_2))\\ 
&=& r(A/B_1\cap B_2) -r(B_1'\cap B_2'/B_1\cap B_2) \\
&=& r(A/B_1\cap B_2) -r(C/B_1\cap B_2)\\
&=& r(A/B_1\cap B_2)-1 > n-1
\end{eqnarray*}
 Then $A \not \blacktriangleleft B'$, $r(A/B') = n-1$ and by the induction hypothesis the lemma is true.
\end{proof}

Now we can prove the main result of the section

\begin{teo}
Let $M$ be a 3-flat matroid. Then $M$ is pseudomodular.
\label{3fwcb}
\end{teo}

\begin{proof}
Suppose $M$ is not have pseudomodular. We will proof that it is not 3-flat. Let $A, B \in \mathcal{F}$ such that $A \not \blacktriangleleft B$ and $r(A/B) = 1$. Then by lemma \ref{wcb1} there are flats $B_1, B_2 \subseteq B$ such that $r(A/B_1) = (A/B_2) = 1 < r(A/B_1\cap B_2)$. Let $F_1 = cl(A\cup B_1)$, $F_2 = cl(A\cup B_2)$ and $F_3 = B$ and consider $\mathcal{C} = \{F_1, F_2, F_3\}$. We have
\begin{itemize}
	\item $F_\emptyset \subseteq cl(B \cup A)$. Then $r(F_\emptyset) = r(B\cup A)$
	\item $F_{\{1,2\}} = cl(A\cup B_1)\cap cl(A\cup B_2) \supseteq (B_1 \cap B_2)\cup A$. Then $r(F_{\{1,2\}}) \geq r((B_1 \cap B_2)\cup A)$
	\item $F_{\{1,3\}} = cl(A\cup B_1)\cap B = B_1$ as $cl(A\cup B_1)\backslash B_1 \subseteq E \backslash B$. Then $r(F_{\{1,3\}}) = r(B_1)$
	\item $F_{\{2,3\}} = cl(A\cup B_2)\cap B = B_2$ as $cl(A\cup B_2)\backslash B_2 \subseteq E \backslash B$. Then $r(F_{\{2,3\}}) = r(B_2)$
	\item $F_{\{1,2,3\}} = B_1\cap B_2$. Then $F_{\{1,2,3\}} = r(B_1\cap B_2)$ 
\end{itemize}
Then
\begin{eqnarray*}
\Delta(C) &=& r(F_\emptyset)- r(F_1)- r(F_2)- r(F_3)+ r(F_{\{1,2\}})\\
					&+& r(F_{\{1,3\}})+ r(F_{\{2,3\}}) - r(F_{\{1,2,3\}})\\
					&\geq & r(B\cup A)-r(B) - r(B_1\cup A) - r(B_2\cup A)\\
					&+& r((B_1 \cap B_2)\cup A)+r(B_1)+r(B_2) -r(B_1\cap B_2)\\
					&=& (r(B\cup A)-r(B)) -(r(B_1\cup A) - r(B_1))\\
					&-& (r(B_2\cup A) - r(B_2))+(r((B_1 \cap B_2)\cup A)-r(B_1\cap B_2))\\
					&=& r(A/B)-r(A/B_1)-r(A/B_2)+r(A/B_1 \cap B_2)\\
					&=& 1 - 1 - 1 +r(A/B_1 \cap B_2)\\
					&=& r(A/B_1 \cap B_2)-1\\
					&>& 0
\end{eqnarray*}
\end{proof}

Note that the lemma was used to ensure $F_{\{1,3\}} =  B_1$ and $F_{\{2,3\}} =  B_2$. If $r(A/B_1) > 1$, then not necessarily $cl(A\cup B_1)\backslash B_1 \subseteq E \backslash B$. We get pseudomodularity of strict gammoids as a corollary. 
\begin{crl}
Let $M$ be a strict gammoid, then $M$ is pseudomodular.
\label{corwcb}
\end{crl} 

The converse of Theorem \ref{3fwcb} is not true. We show the following example of a pseudomodular matroid that is not 3-flat. 

\begin{figure}[ht]
\centering
\caption{}
\scalebox{0.9}{
		\includegraphics{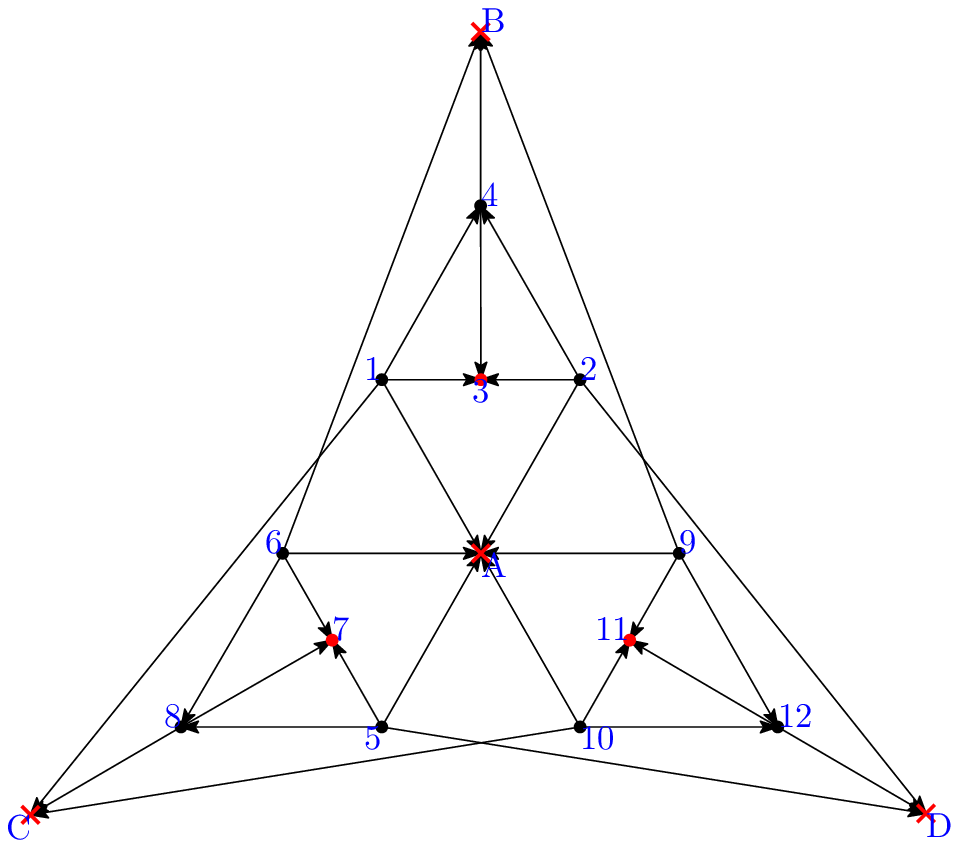}}
	\label{fig:counter1}
\end{figure}

\begin{eje}
Let $N = \{1,2,3,4,5,6,7,8,9,10,11,12\}$ and let $M$ be the matroid on $N$ which has as flats any subset of $N$ of size less or equal to 4, $\{1,3,4,6,7,8\}$, $\{2,3,4,9,11,12\}$, $\{5,7,8,10,11,12\}$, any subset of size 5 not contained in any of the previous 3, $F_1 = \{1,2,3,4,5,6,7,8\}$, $F_2 = \{1,2,3,4,9,10,11,12\}$, $F_3 = \{5,6,7,8,9,10,11,12\}$ and $N$. Figure \ref{fig:counter1} shows a strict gammoid, with sinks $A,B,C,D,3,7,11$ where restricting $\{A,B,C,D\}$ gives us matroid $M$.

Note that if $\mathcal{C} = \{F_1, F_2, F_3\}$, we have $F_{\{1,2\}} = \{1,2,3,4\}$, $F_{\{1,3\}} = \{5,6,7,8\}$, $F_{\{2,3\}} = \{9,10,11,12\}$ and $F_{\{1,2,3\}} = \emptyset$. So $\Delta(\mathcal{C}) = 7-3\cdot6+3\cdot4 = 1$, so $M$ is not 3-flat. Note that for any flat $B$ which is independent we have $A \blacktriangleleft B$ for every $A\in\mathcal{F}$, as $B_0$ would be the set of elements in $B$ that are not coloop in $B\cup A$. So the only flats we need to check are $F_1$, $F_2$, $F_3$, $\{1,3,4,6,7,8\}$, $\{2,3,4,9,11,12\}$ and $\{5,7,8,10,11,12\}$, which are easy to do using lemma \ref{wcb1}. This is our first example of a pseudomodular gammoid which is not a strict gammoid.
\end{eje}

\section{Flatness Degree}

As we have seen being 3-flat implies many interesting properties such as pseudomodularity. It may be worth noting that 3-flatness is sufficient ofr CM-triviality in \cite{hrushovski1993new}. This facts got us interested in studying $n$-flatness in general. We can begin with the following natural definition

\begin{defn}
Given a matroid $M$ we say that $\phi(M)$ is the \textit{flatness degree} of $M$ if $M$ is $\phi(M)$-flat but not $\phi(M)+1$-flat. If such integer does not exist (i. e. $M$ is totally flat) $\phi(M) = \omega$
\end{defn}

As $n$ increases it becomes much harder to study $n$-flatness. However, we will show that for each $n$ there is a matroid with flatness degree $n$, in other words, being $n$-flat is in fact different for each $n$. First we start by presenting some tools that are useful for studying the $\Delta$ function. Proposition \ref{repet} was already shown by Holland in \cite{holland1999flatness}, which includes a much deeper study of the $\Delta$ function (as $T$ function).

\begin{prop}
\label{repet}
Let $\mathcal{C} = \{F_i \in \mathcal{F} \enspace | \enspace k \in I\}$ a collection of flats. Suppose that there are $F_i, F_j \in \mathcal{C}$ such that $F_i \subseteq F_j$. Then $\Delta(C) = \Delta(C\backslash\{F_i\})$
\end{prop}

\begin{proof}
Note that $\bigcup C = \bigcup (C\backslash\{F_i\})$. Let $\mathcal{C}_i = \{F_{\{i, k\}} \enspace | \enspace k \in I \enspace \wedge \enspace i \neq k\}$. So $\Delta(C) = \Delta(C\backslash\{F_i\})-\Delta(\mathcal{C}_i)$. But for each $S \subseteq I \backslash\{i,j\}$, $F_{S\cup\{i\}} = F_{S\cup\{i, j\}}$ so $\Delta(\mathcal{C}_i) = 0$ and we have the desired result.
\end{proof}

\begin{prop}
\label{cyclic}
Let $\mathcal{C} = \{F_i \in \mathcal{F} \enspace | \enspace k \in I\}$ be a collection of flats. Then there exists $\mathcal{C}'$ such that $|\mathcal{C}'| = |\mathcal{C}|$, $\Delta(\mathcal{C}') \geq \Delta(\mathcal{C})$ and all elements of $\mathcal{C}'$ are cyclic.
\end{prop}

Although this proposition was not previously stated as it is, it is widely regarded that being transversal (as well as being strict gammoid and hence being totally flat) depends only on the cyclic flats. So the proposition results natural.
\begin{proof}
The proof is done by induction on the sum of the number of isthmus of elements in $\mathcal{C}$. If the sum is $0$ we have the desired result. Now suppose $F_i \in \mathcal{C}$ has an isthmus $e$.  Let $S := \{i\in I \enspace | \enspace e \in F_i\}$. First suppose $e$ is also an isthmus in $F_\emptyset$. Then every subset of $F_\emptyset$ who contains $e$, has it as an isthmus. Now for every $j \in S$, $F_j\backslash\{e\}$ is also a flat different than $F_k$ for every $k \neq j$. Consider $\mathcal{D} := \{F_j \enspace | \enspace j \notin S\} \cup \{F_k\backslash\{e\} \enspace | \enspace k \in S\}$. We have $|\mathcal{D}| = |\mathcal{C}|$, and $\Delta(\mathcal{C})-\Delta(\mathcal{D}) =  \sum\limits_{A \subseteq S} (-1)^{|A|} = 0$. As $\mathcal{D}$ has $|S|$ less isthmuses than $\mathcal{C}$ by induction hypothesis there is a collection of cyclic flats $\mathcal{C}'$ such that $|\mathcal{C}'| = |\mathcal{D}| = |\mathcal{C}|$ and $\Delta(\mathcal{C}') \geq \Delta(\mathcal{D}) = \Delta(\mathcal{C})$ 

Now suppose $e$ is not an isthmus in $F_\emptyset$. Let $\mathcal{D} := \{F_j \enspace | \enspace j \neq i\} \cup \{F_i\backslash\{e\} \}$. Again $F_i \backslash\{e\} \neq F_j$ for $i \neq j$ and $|\mathcal{D}| = |\mathcal{C}|$. If $|S| \geq 2$, have $\Delta(\mathcal{C})-\Delta(\mathcal{D}) =  \sum\limits_{A \subseteq (S\backslash \{i\})} (-1)^{|A|+1} = 0$. If $S = \{i\}$ then $\Delta(\mathcal{C})-\Delta(\mathcal{D}) = -1$. The number of isthmuses in $\mathcal{D}$ is one less than $\mathcal{C}$ so by induction hypothesis the proposition such $\mathcal{C}'$ exists.
\end{proof}

\begin{prop}
\label{cont}
Let $\mathcal{C} = \{F_i \in \mathcal{F} \enspace | \enspace k \in I\}$ be a collection of flats such that $|\mathcal{C}| > 1$. Then there exists $\mathcal{C}'$ such that $|\mathcal{C}'| = |\mathcal{C}|$, $\Delta(\mathcal{C}') \geq \Delta(\mathcal{C})$ and $\forall F \in \mathcal{C}'$ $F \subseteq cl(\bigcup(\mathcal{C}' \backslash \{F\}))$ 
\end{prop}
\begin{proof}

Let $s(\mathcal{C}) = \sum\limits_{i \in I} r(F_i\backslash cl(\bigcup\limits_{j\neq i} F_j))$. If $s(\mathcal{C}) = 0$, then clearly $\forall i \in I \enspace F_i \subseteq cl(\bigcup\limits_{j\neq i} F_j)$. Now we will proceed by induction on $s(\mathcal{C})$. Suppose $s(\mathcal{C}) > 0$. Then there exists $i \in I$ such that $r(F_i\backslash cl(\bigcup\limits_{j\neq i} F_j)) \geq 1$. Let $x \in F_i\backslash cl(\bigcup\limits_{j\neq i} F_j)$ non loop, and take any $j \in I$ $j \neq i$. Consider $F_j' = cl(F_j\cup\{x\})$ and $\mathcal{D} = \{F_k \enspace | k \in I \enspace k\neq j\} \cup \{F_j'\}$. Clearly $|\mathcal{C}| = |\mathcal{D}|$. Note that if $A \subseteq I\backslash\{i,j\}$ we have $F_j' \cap F_A = F_{A\cup \{j\}}$. As $x\in (F_j'\cap F_i) \backslash (F_j\cap F_i)$ we have
\begin{eqnarray*}
	\Delta(\mathcal{D}) - \Delta(\mathcal{C}) &=& -r(F_j')+r(F_j'\cap F_i) +r(F_j) - r(F_j\cap F_i) \\
	&=& r(F_j'\cap F_i)- r(F_j\cap F_i)+1 \\
	&\geq& 0
\end{eqnarray*}

But now $r(F_i\backslash cl(\bigcup\limits_{k\neq i, j} F_k \cup \{F_j'\})) < r(F_i\backslash cl(\bigcup\limits_{k\neq i} F_k))$, so $s(\mathcal{D}) < s(\mathcal{C})$. Then by induction hypothesis, there is a $\mathcal{C}'$ such that $|\mathcal{C}'| = |\mathcal{D}|$ and $\Delta(\mathcal{C}') \geq \Delta(\mathcal{D}) \geq \Delta(\mathcal{C})$ and $\forall F \in \mathcal{C}'$ $F \subseteq cl(\bigcup(\mathcal{C}' \backslash \{F\}))$.
\end{proof}

It is now easy to verify that one can assume that all elements of $\mathcal{C}$ must be of rank at least 2. So we can focus on studying only the particular class of collections of flats who are cyclic, with rank at least 2, no flat is contained in another, and every flat is contained in the closure of the union of the rest of the flats. Before we show matroids with fixed flatness degree, we prove the following combinatorial lemma. 

\begin{lema}
\label{propcomb}
Let $n$ be a positive integer, and $l$, $m$ positive integers lesser than $n$. Then
\begin{equation*}
\sum\limits_{k = 0}^m (-1)^k \dbinom{n-k}{l} \dbinom{m}{k} = \dbinom{n-m}{l-m}
\end{equation*}
In particular, when $m>l$ the equation is equal to 0.
\end{lema}

\begin{proof}
Note that the $\dbinom{n-k}{l}\dbinom{m}{k}$ is the number of ways of choosing a subset $K$ of $[m]$ of cardinality $k$ and then choosing a subset $L$ of $[n]$ of cardinality $l$ such that $L\cap K = \emptyset$. The LHS is summing the number of ways to take an arbitrary subset $K \subseteq [m]$ and then choosing $L$ with an alternating sign depending on $|K|$. Now if we choose first the set $L$, we have now to choose subsets $K$ of $[m]\backslash L$. So for a fixed set $L$ where $j = |[m]\backslash L|$ we will have $\sum\limits_{k = 0}^{j}\dbinom{j}{k}(-1)^j$. Now this is 0 for $j>1$ and 1 for $j= 0$. So the only terms that survive are the ones where $[m] \subseteq L$. But the number of ways of choosing $L$ such that $[m] \subseteq L \subseteq [n]$ is precisely $\dbinom{n-m}{l-m}$
\end{proof}

Now we have the tools to prove the following theorem

\begin{teo}
\label{nflat}
Let $n \geq 2$. Then there exists a matroid $M$ with flatness degree $n$. Moreover, there exists a gammoid with flatness degree $n$.
\end{teo}

\begin{proof}
Consider $N := \dbinom{ [n] }{2}$, the subsets of $[n] := \{1,2, \dots n\}$ of size 2. For $i\in [n]$ let $A_i = \{x \in N \enspace | \enspace i \in x\} $. Consider the transversal matroid $M$ on the set $N$ given by the presentation $(A_1, A_2, \dots, A_n, N, N \dots N)$, where there are $\dbinom{ n-1}{2}-n$ copies of $N$ in $\mathcal{A}$. Note that as $M$ is transversal, it is a gammoid. We claim that $\phi(M) = n$.

Note that any subset of $N$ with cardinality  $\dbinom{ n-1}{2}-n$ or less is independent. So to get dependent sets in $M$ we need to look at subsets of cardinality at least $\dbinom{n-1}{2}-n+1$. But by considering such big subsets we will end up with elements that can be sent to $A_i$ for at least $n-2$ different $i$'s in $[n]$. But then we would need subsets of size at least $\dbinom{ n-1}{2}-1$ which would necessarily have access $A_i$ for at least $n-1$ different $i$'s in $[n]$. So the smallest dependent set would be $F_i := \{x \in N \enspace | \enspace i \notin x\}$ for $i \in [n]$. In fact $M$ is the matroid generated by having $\{F_i \enspace | \enspace i \in [n]\}$ as its set of circuits. This could be an easier definition for $M$ (which even works for $n = 3$) but we wanted to show that is in fact a gammoid. We name them $F_i$ and not $C_i$ because they are also the only cyclic flats apart from $E$ itself.

So to check that $M$ is $k$-flat for a given $m$, by proposition \ref{cyclic} we need only to focus on collections of $F_i$ as they are the only non trivial cyclic flats of $M$. Let $A \subseteq [n]$ of size $m > 2$. Consider $\mathcal{C} := \{F_i \enspace | \enspace i \in A\}$. If $B \subseteq A$ with $|B| > 2$, then we have that $F_B = \{x \in N \enspace | \enspace  x\cap B = \emptyset \}$. So $F_B$ is independent and $r(F_B) = |F_B| = \dbinom{n-|B|}{2}$. As $F_i$ is a circuit for every $i \in [n]$, we have $r(F_i) = \dbinom{n-1}{2}-1$. Finally, $r(F_\emptyset) = r(N) = \dbinom{n-1}{2}$. So using lemma \ref{propcomb} we have

\begin{eqnarray*}
\Delta(\mathcal{C}) &=& \sum\limits_{k = 2}^m \dbinom{n-k}{2}\dbinom{m}{k} -m\left(\dbinom{n-1}{2} -1 \right)+\dbinom{n-1}{2}\\
&=& \sum\limits_{k = 0}^m \dbinom{n-k}{2}\dbinom{m}{k} + m+\dbinom{n-1}{2} - \dbinom{n}{2}\\
&=& m+\dbinom{n-1}{2} - \dbinom{n}{2}\\
&=& m-n+1\\
\end{eqnarray*}

So $\Delta(\mathcal{C})>0$ if and only if $m = n$. So $\phi(M) = n$. Note that the same matroid structure defined by circuits can be used to construct a matroid of flatness degree 3. However this matroid would be the matroid generated by graph $K_4$, which Mason proved in \cite{mason1972class} not to be a gammoid. However example 2 in \cite{mason1972class} is an example of a gammoid with flatness degree 3.
\end{proof}

Having differentiated all of this properties we see now that this is an infinite chain of stronger properties
	\[
	\text{Pseudomodular} \supset \text{3-Flat} \supset \text{4-Flat} \supset \dots \supset \text{Totally Flat}
\]

All of the inclusions are strict, even in the class of gammoids, with the examples shown above. Note that all of the properties are conditions on the lattice of flats. Modularity, which is also a condition on the lattice of flats, implies 3-flatness by Proposition \ref{cont}. However, it does not imply 4-flatness (take 4 planes in $\mathbb{R}^3$ which intersect in different lines).

\section{Problems}

A full algebraic matroid consists of an algebraically closed field $K$ where the rank function is the transcendence degree over a subfield $F$ of $K$. Full algebraic matroids and strict gammoids appear to have some interesting similarities, as they are both pseudomodular \cite{dress1987some}, they are closed under contraction but not restriction and gammoids are actually algebraic \cite{mason1972class}. It may be also worth noting that when Ingleton and Main first proved the existence of non-algebraic matroids in \cite{ingleton1975non} they relied on the fact that whenever we have points as in Example \ref{matroidY}, the point $X$ must exist for full algebraic matroids. As strict gammoids are characterized by flatness and pseudomodularity is a step towards 3-flatness we ask the following question 

\begin{prob}
What is the minimum possible flatness degree for a full algebraic matroid?
\end{prob}

If a matroid $M$ has flatness degree $n$ for a large finite $n$, it seems that $M$ must have necessarily a large rank. On the other hand, for rank 3 matroids flatness degree 3 is possible ($K_4$) but not flatness degree 2. However, as there are no non strict gammoids of rank 3 \cite{ingleton1973gammoids}, the only possible flatness degree for a gammoid of rank 3 is $\omega$. So the bounds for flatness degree for gammoids and matroids in general are different.

\begin{prob}
For a given integer $n>1$, give bounds in terms of rank for a matroid $M$ such that $\phi(M) = n$. Give bounds in terms of rank for a gammoid $M$ such that $\phi(M) = n$.
\end{prob}

The fact that these bounds are different may help the widely known open problem

\begin{prob}
Give an algorithm to determine whether or not a given matroid is a gammoid.
\end{prob}

\break

\nocite{*}
\bibliographystyle{plain}
\bibliography{biblo}

\begin{thebibliography}{10}

\bibitem{bjorner1987pseudomodular}
Anders Bj{\"o}rner and L{\'a}szl{\'o} Lovasz.
\newblock {\em Pseudomodular lattices and continuous matroids}.
\newblock Univ., matem. inst., 1987.

\bibitem{brualdi1987introduction}
Richard~A Brualdi.
\newblock Introduction to matching theory.
\newblock {\em Combinatorial Geometries}, 29:53, 1987.

\bibitem{brualdi1987transversal}
Richard~A Brualdi.
\newblock Transversal matroids.
\newblock {\em Combinatorial Geometries}, 29:72--97, 1987.

\bibitem{dress1987some}
A~Dress and L{\'a}szl{\'o} Lov{\'a}sz.
\newblock On some combinatorial properties of algebraic matroids.
\newblock {\em Combinatorica}, 7(1):39--48, 1987.

\bibitem{evans2011matroid}
David~M Evans.
\newblock Matroid theory and hrushovski's predimension construction.
\newblock {\em arXiv preprint arXiv:1105.3822}, 2011.

\bibitem{hochstattler1989matroid}
Winfried Hochst{\"a}ttler and Walter Kern.
\newblock Matroid matching in pseudomodular lattices.
\newblock {\em Combinatorica}, 9(2):145--152, 1989.

\bibitem{holland1999flatness}
KL~Holland.
\newblock Flatness, convexity and notions of freeness in combinatorial
  geometries.
\newblock {\em algebra universalis}, 41(1):1--21, 1999.

\bibitem{hrushovski1993new}
Ehud Hrushovski.
\newblock A new strongly minimal set.
\newblock {\em Annals of Pure and Applied Logic}, 62(2):147--166, 1993.

\bibitem{ingleton1977transversal}
AW~Ingleton.
\newblock Transversal matroids and related structures.
\newblock In {\em Higher Combinatorics}, pages 117--131. Springer, 1977.

\bibitem{ingleton1975non}
AW~Ingleton and RA~Main.
\newblock Non-algebraic matroids exist.
\newblock {\em Bulletin of the London Mathematical Society}, 7(2):144--146,
  1975.

\bibitem{ingleton1973gammoids}
AW~Ingleton and MJ~Piff.
\newblock Gammoids and transversal matroids.
\newblock {\em Journal of Combinatorial Theory, Series B}, 15(1):51--68, 1973.

\bibitem{mason1971characterization}
JH~Mason.
\newblock A characterization of transversal independence spaces.
\newblock In {\em Th{\'e}orie des matro{\"\i}des}, pages 86--94. Springer,
  1971.

\bibitem{mason1972class}
JH~Mason.
\newblock On a class of matroids arising from paths in graphs.
\newblock {\em Proceedings of the London Mathematical Society}, 3(1):55--74,
  1972.

\bibitem{oxley1992matroid}
James~G Oxley.
\newblock {\em Matroid theory}.
\newblock Oxford university press, 1992.

\end{thebibliography}

\end{document}